\documentclass[11pt,epsf]{article}
\usepackage{amsfonts}
\usepackage{amsmath,amscd}
\usepackage{amssymb}
\usepackage{amsthm}
\usepackage{newlfont}
\newcommand{\f}{\frac}

 \newtheorem{thm}{Theorem}[section]
 \newtheorem{cor}[thm]{Corollary}
 \newtheorem{lem}[thm]{Lemma}
 \newtheorem{prop}[thm]{Proposition}
 \theoremstyle{definition}
 
 \theoremstyle{remark}

 \numberwithin{equation}{section}
\begin{document}
\title{A note on the Schur multiplier of a nilpotent Lie algebra}
\author{Peyman Niroomand$^{a}$, Francesco Russo$^{b}$ \\
{\small $^{a}$ School of Mathematics and Computer Science
Damghan University of Basic Sciences,}\\{\small Damghan, Iran.}\\{\small $^b$ University of Naples Federico II
via Cinthia, 80126,  Naples, Italy.}}

\date{ } \maketitle

\begin{abstract}
For a nilpotent Lie algebra $L$ of dimension $n$ and dim$(L^2)=m~(m\geq 1)$,
we find the upper bound dim$(M(L))\leq \frac
{1}{2}(n+m-2)(n-m-1)+1$, where $M(L)$ denotes the Schur multiplier
of $L$. In case $m=1$ the equality holds if and only if $L\cong
H(1)\oplus A$, where $A$ is an abelian Lie algebra of dimension
$n-3$ and $H(1)$ is the Heisenberg algebra of dimension 3. 

\end{abstract}

\section{Introduction}
It is well known that restrictions on the Schur multiplier of finite
$p$--groups ($p$ a prime) are related to significant information on
the Schur multiplier $M(L)$ of a nilpotent Lie algebra $L$ of
dimension $n$. Many times it is possible to get structural results
on $L$ only looking at the size of $M(L)$. This fact was already
noted by Batten and others in \cite[Theorem 5]{es}. Given $t(L) =
\frac{1}{2}n(n- 1) - \mathrm{dim} (M(L))$, they prove  in
\cite[Theorem 3]{es} that $t(L) = 1$ if and only if $L\cong H(1)$,
where $H(1)$ is the Heisenberg algebra of dimension 3. Moreover
\cite[Theorem 5]{es} shows that $t(L) = 2$ if and only if $L \cong
H(1)\oplus A $, where $A$ is an abelian algebra of $\mathrm{dim}(A)
= 1$.

To convenience of the reader, we recall that a finite dimensional
Lie algebra $L$ is called $Heisenberg$ provided that $L^2 = Z(L)$
and $\mathrm{dim}(L^2) = 1$. Such algebras are odd dimensional with
basis $v_1, \ldots , v_{2m}, v$ and the only non-zero multiplication
between basis elements is $[v_{2i-1}, v_{2i}] = - [v_{2i},
v_{2i-1}]= v$ for $i = 1, \ldots ,m$. The symbol $H(m)$ denotes the
Heisenberg algebra of dimension $2m + 1$.

There are successive contributions on the same line of
investigation, once we prescribe a value for $t(L)$. In \cite{es3}
the cases $t(L)=3,4,5,6$ are studied and weaker characterizations
are obtained (see \cite[Theorems 2, 3, 4]{es3}). Under the same
prospective we should read \cite{es2,mo}, where homological
machineries are involved.

It is instructive to note that in \cite[Section 1]{es} Batten and
others declare explicitly that their contributions originated from
the classification of Zhou in \cite{zhou}, where a corresponding
situation for $p$--groups was analyzed. In a certain sense the same
motivation allows us to write the present paper.

A classic restriction of Jones \cite[Theorem 3.1.4]{ka} on the Schur
multiplier of a non-abelian $p$--group has been recently improved by
the first author. More precisely it is proved in \cite{ni} that a
non-abelian $p$--group $G$ of order $p^n$ with derived subgroup of
order $p^k$ has \[|M(G)| \leq p^{\frac{1}{2} (n+k-2)(n-k-1)+1}.\] In
particular,
\[|M(G)| \leq p^{\frac{1}{2} (n-1)(n-2)+1}\] and the equality holds
in this last bound if and only if $G = H \times Z$, where $H$ is
extra special of order $p^3$ and exponent $p$, and $Z$ is an
elementary abelian $p$--group.

The present paper is devoted to obtain similar results for Lie
algebras.

\section{Preliminaries}
The present section illustrates how the ideas of Zhou \cite{zhou}
have been adapted in \cite{es3} to the context of Lie algebras. This
is an important feedback for our main theorems. For instance, the
following is analogous to \cite[Theorem 2.5.2]{ka}.
\begin{prop} $\mathrm{(See}$ \cite[Lemma 4]{es}$\mathrm{).}$ \label{p}Let $L$ be a finite dimensional Lie algebra, $K$ an ideal of
$L$ and $H = L/K$. Then there exists a finite dimensional Lie
algebra $J$ and and ideal $M$ of $J$ such that
\begin{itemize} \item[(i)]$L^2\cap K\cong J/M;$
\item[(ii)]$M(L)\cong M;$ \item[(iii)]$M(H)$ is an epimorphic image of $J$.
\end{itemize}
\end{prop}

The Schur multiplier of the direct product of two finite groups is
equal to the direct product of the Schur multipliers of the two
factors plus the tensor product of the abelianization of the two
groups (see \cite[Theorem 2.2.10]{ka}). This is a general fact of
homology, which is known as the K\"unneth Formula (see \cite{rot}),
and it is true also for two finite dimensional Lie algebras $H$ and
$K$. The K\"unneth Formula was originally obtained by Schur in 1904
(see \cite{sch}). We recall that the symbol $\otimes$ denotes the
usual tensor product of abelian Lie algebras. Then we have
\[M(H\oplus K)=M(H)\oplus M(K) \oplus (H/H^2\otimes K/K^2).\] At this point we
may give a short proof of \cite[Theorem 1]{es} as follows.

\begin{thm}\label{ds}
Let $A$ and $B$ be finite dimensional Lie algebras. Then
\[\mathrm{dim} (M(A\oplus B)) = \mathrm{dim} (M(A)) + \mathrm{dim} (M(B))
+ \mathrm{dim}(A/A^2\otimes B/B^2).\]
\end{thm}

\begin{proof} Use the K\"unneth
Formula above mentioned.
\end{proof}

The following result is proved in \cite[Theorem 2.5.5 (ii)]{ka} for
groups.
\begin{cor}\label{sr}
Let $L$ be a finite dimensional Lie algebra, $K$ an ideal of $L$ and
$H=L/K$. Then
\[\mathrm{dim}(M(L))+\mathrm{dim}(L^2\cap K)\leq \mathrm{dim}(M(H))+\mathrm{dim}(M(K))+\mathrm{dim}(H/H^2\otimes K/K^2).\]
%Moreover, if $H$ and $K$ are abelian, then
%\[\mathrm{dim}(M(L))+\mathrm{dim}(L^2\cap K)\leq \mathrm{dim}(M(H))+\mathrm{dim}(M(K))+\mathrm{dim}(H\otimes K).\]
\end{cor}
\begin{proof}
From Theorem \ref{ds} it is enough to prove that
$\mathrm{dim}(M(L))+\mathrm{dim}(L^2\cap K)\leq
\mathrm{dim}(M(H\oplus K))$. From Proposition \ref{p},
\[\mathrm{dim}(M(L))+\mathrm{dim}(L^2\cap K)=\mathrm{dim}(M)+\mathrm{dim}(J/M)=\mathrm{dim}(J)=\mathrm{dim}(\epsilon (M(H)))\]
for a suitable epimorphism $\epsilon$ of Lie algebras. Since
$\epsilon$ sends systems of generators in systems of generators,
$\mathrm{dim}(\epsilon (M(H))\leq  \mathrm{dim} (M(H\otimes K))$, as
claimed.
\end{proof}

Now we mention three results from \cite{es,es2,es3,mo} to
convenience of the reader.

\begin{lem}$\mathrm{(See}$\cite[Example 3]{es2} $\mathrm{and}$ \cite[Theorem 24]{mo}$\mathrm{).}$\label{h}
\begin{itemize}
\item[(i)]$\mathrm{dim}(M(H(1)))=2$.
\item[(ii)]$\mathrm{dim}(M(H(m)))=2m^2-m-1$ for all $m\geq 2$.
\end{itemize}
\end{lem}

\begin{lem} $\mathrm{(See}$\cite[Lemma 3]{es}$\mathrm{).}$\label{ab} A Lie algebra $L$ of dimension $n$ is abelian if and only if $\mathrm{dim}(M(L))=\frac{1}{2}n(n-1)$.
\end{lem}

\begin{prop} $\mathrm{(See}$\cite[Proposition 1]{es3}$\mathrm{).}$\label{lj} A nilpotent Lie algebra $L$ of dimension
$n$ has $\mathrm{dim}(L^2)+ \mathrm{dim} (M(L))\leq
\frac{1}{2}n(n-1)$.
\end{prop}

Hardy and others declare in \cite{es3} that they were working on
higher values of $t(L)$ (namely, the case $t(L)\geq9$ is still
open). Their investigations come from the bound in Proposition
\ref{lj}.% Of course, changing the bound, we will have a different
%classification of the nilpotent Lie algebras in terms of $t(L)$.
%More details will be given in the next section.
%However their
%classification is invariant not only under isomorphisms of Lie
%algebras, but also under isoclinisms of Lie algebras. This notion
%was already studied in \cite{mo} and is recalled below. Originally
%it was introduced by P.Hall in the context of finite $p$--groups.

% \begin{defn} \label{d:1}
% Let $L_1$ and $L_2$ be two finite dimensional Lie algebras.
% A pair $(\alpha,\beta)$ is said to be an isoclinism
%from $L_1$ to $L_2$, or briefly $L_1 \ _{\widetilde{}} \ L_2$, if
% the following conditions are satisfied:
%\begin{itemize}
%\item[(i)] $\alpha$ is an isomorphism of Lie algebras from
%$L_1/Z(L_1)$ to $L_2/Z(L_2)$; \item[(ii)]  $\beta$ is an isomorphism
%of Lie algebras from $[L_1,L_1]$ to $[L_2,L_2]$;
%\item[(iii)]  The following diagram is commutative:
%\[\begin{CD}
%\frac {L_1}{Z(L_1)}\oplus \frac {L_1}{Z(L_1)}  @>\alpha^{2}>>
%\frac {L_2}{Z(L_2)} \oplus \frac {L_2}{Z(L_2)}\\
%@V\gamma(L_1) VV @V\gamma(L_2)VV \vspace{.3cm}\\
%[L_1,L_1] @>\beta>>[L_2,L_2].
%\end{CD}\]
%where $\gamma (L_1)(x_1Z(L_1),y_1Z(L_1))=[x_1,y_1]$ and $\gamma
%(L_2)(x_2Z(L_2),y_2Z(L_2))=[x_2,y_2]$ for all $x_1,y_1 \in L_1$ and
%$x_2,y_2 \in L_2$.
%\end{itemize}
%\end{defn}

\section{Main Theorem}

The following result provides a bound which is less than the bound
in Proposition \ref{lj} except for the case $L\cong H(1)$.

\begin{thm}\label{mt}
Let $L$ be a nilpotent Lie algebra of $\mathrm{dim}(L)=n$ and
$\mathrm{dim}(L^2)=m~(m\geq 1)$. Then\[\mathrm{dim}~M(L)\leq \frac
{1}{2}(n+m-2)(n-m-1)+1.\] Moreover, if $m=1$, then the equality
holds if and only if $L\cong H(1)\oplus A$, where $A$ is an abelian
Lie algebra of  $\mathrm{dim}(A)=n-3$.
\end{thm}
\begin{proof}
Assume $\mathrm{dim}(L^2)=1$. Then $L/L^2$ is an abelian Lie algebra
of $\mathrm{dim}(L/L^2)=n-1$. Since $L^2\subseteq Z(L)$, we may
consider a complement $H/L^2$ of $Z(L)/L^2$ in $L/L^2$. So we have
$L=H+Z(L)$ and $L^2=H^2$. On the other hand, $H\cap Z(L)\subseteq H$
and so $Z(H)=L^2$. Since $L^2\subseteq Z(L)$, $L^2$ must have a
complement $K$ in $Z(L)$. Let $L^2\oplus K=Z(L)$ so we have $L\cong
K\oplus H$. By Theorem \ref{ds},
\[\mathrm{dim} (M(K\oplus H)) = \mathrm{dim} (M(K)) + \mathrm{dim} (M(H)) + \mathrm{dim} (K/K^2\otimes H/H^2).\]
Since $H$ is a Heisenberg algebra and $K$ is abelian, two cases
should be considered.

Case 1. Assume $\mathrm{dim}(H)=2m+1$ for $m\geq 2$. By Lemmas
\ref{h} and \ref{ab},
\[\mathrm{dim}(M(H))=2m^2-m-1,\]
\[\mathrm{dim}(M(K))=\frac{1}{2}(n-2m-1)(n-2m-2),\]
\[\mathrm{dim}(K \otimes H/H^2)=\mathrm{dim}(K) \cdot \mathrm{dim}(H/H^2)=(n-2m-1)(2m).\]
and we deduce that
\[\mathrm{dim} (M(K\oplus H))=\f{1}{2}(n-2m-1)(n-2m-2)+(2m^2-m-1)+(n-2m-1)(2m)\]
%\[=\f{n^2-2mn-2n-2mn+4m^2+4m-n+2m+2+4m^2-2m-2+4mn-8m^2-8m}{2}\]
%\[=\f{n^2-3n-4m}{2}<\f{n^2-3n+4}{2}=\f{1}{2}(n-1)(n-2)+1.\]
\[=\f{1}{2}n(n-3)<\f{1}{2}(n-1)(n-2)+1.\]

Case 2. Assume $m=1$ and $H\cong H(1)$. By Lemmas \ref{h} and
\ref{ab},
\[\mathrm{dim}(M(H))=2,\]
\[\mathrm{dim}(M(K))=\frac{1}{2}(n-3)(n-4),\]
\[\mathrm{dim}(K \otimes H/H^2)=\mathrm{dim}(K) \cdot \mathrm{dim}(H/H^2)=(n-3)\cdot (2)\]
and  we deduce that \[\mathrm{dim}(M(K\oplus
H))=\f{1}{2}(n-3)(n-4)+2+ 2(n-3)\]%=\frac{n^2-7n+12+2+4n-16}{2}\]
%\[=\f{n^2-3n-2}{2}<\f{n^2-3n+4}{2}=\f{1}{2}(n-1)(n-2)+1.\]
\[=\f{1}{2}n(n-3)+2.\]Now we proceed by induction on $m$.

Let $m>1$. Since $0\neq L^2\cap Z(L)$, there exists  an ideal $K$ of
dimension $1$ contained in $L^2\cap Z(L)$. By induction hypothesis
and Lemma \ref{sr}, we have
\[1+\mathrm{dim}(M(L))\leq \mathrm{dim}(M(L/K))+\mathrm{dim}(M(K))+\mathrm{dim}(L/L^2\otimes K)\]
and so
\[\mathrm{dim}(M(L))\leq \f{1}{2}(n+m-4)(n-m-1)+1+n-m-1\]
\[=\f{1}{2}(n+m-2)(n-m-1)+1,\] as claimed.
\end{proof}


\begin{thebibliography}{20}
%\bibitem{ams} V. Alamian, H. Mohammadzadeh and A. R. Salemkar, Some properties of the Schur multiplier and covers of Lie algebras, \textit{Comm. Algebra} \textbf{36} (2008), 697--707.
\bibitem{es}  P. Batten, K. Moneyhun and E. Stitzinger, On characterizing nilpotent Lie algebras by their multipliers, \textit{Comm. Algebra} \textbf{24} (1996), 4319--4330.
\bibitem{es2} P. Batten and E. Stitzinger, On covers of Lie algebras, \textit{Comm. Algebra} \textbf{24}, (1996), 4301--4317.
\bibitem{es3} P. Hardy and E. Stitzinger, On characterizing nilpotent Lie algebras by their multipliers $t(L) = 3, 4, 5, 6$, \textit{Comm. Algebra} \textbf{26} (1998), 3527–--3539.
\bibitem{ka}  G. Karpilovsky, \textit{The Schur multiplier}, London Math. Soc. Monogr. (N.S.) 2, London, 1987.
\bibitem{mo}  K. Moneyhun, Isoclinisms in Lie algebras, \textit{Algebras Groups Geom.} \textbf{11} (1994), 9--22.
\bibitem{ni}  P. Niroomand, On the order of Schur multiplier of non-abelian $p$--groups, \textit{J. Algebra} \textbf{322} (2009), 4479--4482.
\bibitem{rot} J. Rotman, \textit{An Introduction to Homological Algebra}, Academic Press, San Diego, 1979.
\bibitem{sch} I. Schur, Untersuchungen über die Darstellung der endlichen Gruppen durch gebrochene lineare Substitutionen, \textit{J. Reine Angew. Math.} \textbf{132} (1907), 85-–-137
\bibitem{zhou}X. Zhou, On the order of the Schur multiplier of finite $p$--groups, {\it Comm. Algebra} {\bf 22} (1994), 1--8.
\end{thebibliography}
\end{document}